\documentclass{amsart}
\usepackage{amsthm,amssymb,amsfonts,amsmath,mathrsfs}
\usepackage{enumerate}
\usepackage[nodvipsnames]{xcolor}
\usepackage{soul}
\usepackage{color}
\usepackage{eurosym}
\usepackage{verbatim}
\usepackage{graphicx}
\usepackage{marginnote}
\usepackage{hyperref}
\usepackage[T1]{fontenc}

\newtheorem{theorem}{Theorem}

\newtheorem{corollary}[theorem]{Corollary}

\newtheorem{definition}[theorem]{Definition}

\newtheorem{lemma}[theorem]{Lemma}

\newtheorem{proposition}[theorem]{Proposition}
\newtheorem{remark}[theorem]{Remark}

\input{tcilatex}


\begin{document}

\title[Quantitative statistical stability and linear response ...]{Quantitative statistical stability and linear response for irrational
rotations and diffeomorphisms of the circle}
\author{Stefano Galatolo} 
\address{
Dipartimento di Matematica, Universit\`a di Pisa, Largo Bruno Pontecorvo 5,
56127 Pisa, Italy email: stefano.galatolo@unipi.it}
 \author{Alfonso Sorrentino} 
\address{
Dipartimento di Matematica, Universit\`a degli Studi di Roma ``Tor Vergata",
Via della ricerca scientifica 1, 00133, Roma, Italy. Email:
sorrentino@mat.uniroma2.it} 
\date{\today }
\maketitle

\begin{abstract}
 We prove quantitative
statistical stability results for a large class of small $C^{0}$ perturbations \ of
circle diffeomorphisms with irrational rotation numbers. We show that if the rotation number is Diophantine the
invariant measure varies in a H\"older way under perturbation of the map and
the H\"older exponent depends on the Diophantine type of the rotation number.
The set of admissible perturbations includes the ones coming from spatial discretization and hence numerical truncation.
We also show linear response for smooth perturbations that preserve the
rotation number, as well as for more general ones. 
This is done {by means of} classical tools from KAM theory,
while the quantitative stability results are obtained by transfer operator
techniques applied to suitable spaces of measures with a weak topology.
\end{abstract}

\section{Introduction}

Understanding the statistical properties of a certain dynamical system is of
fundamental importance in many problems coming from pure and applied
mathematics, as well as in developing applications to other sciences.

\medskip

In this article, we will focus on the concept of \textit{statistical
stability} of a dynamical system, \textit{i.e.}, how its statistical
features change when the systems is perturbed or modified. The interest in
this question is clearly motivated by the need of controlling how much, and
to which extent, approximations, external perturbations and uncertainties can affect the qualitative
and quantitative analysis of its dynamics.

\medskip

Statistical properties of the long-term evolution of a system are reflected,
for instance, by the properties of its invariant measures. When the system
is perturbed, it is then useful to understand, and be able to predict, how
the relevant\footnote{%
The concept of \textit{relevant} is strictly related to the analysis that is
carried out. Hereafter, we will be interested in so called \emph{\ physical
measures} (see footnote \ref{notap} or \cite{Y}). In other contexts, other
kinds of measures might be considered, for example, the so-called measures
of maximal entropy.} invariant measures change by the
effect of the perturbation, \textit{i.e.}, what is called the \textit{%
response} of the system to the perturbation. In particular, it becomes
important to get quantitative estimates on their change by effect of the
perturbation, as well as understanding the \textit{regularity} of their
behavior, for instance differentiability, Lipschitz or H\"{o}lder
dependence, etc...

\medskip

These ideas can be applied to many kinds of systems and these concepts can
be studied in many different ways. In this paper we will consider \emph{%
discrete deterministic dynamical systems} and \emph{\ deterministic
perturbations. }

\smallskip

More specifically, we will consider systems of the kind $(X,T_{0})$, where $%
X $ is a compact metric space and $T_{0}:X\rightarrow X$ a map, whose
iterations determine the dynamics; we investigate perturbed systems $%
\{(X,T_{\delta })\}_{\delta \in \lbrack 0,\overline{\delta })}$, where $%
T_{\delta }:X\rightarrow X$ are such that $T_{\delta }\rightarrow T_{0}$, as 
$\delta \rightarrow 0$, in some suitable topology.

\smallskip

For each $\delta \in \lbrack 0,\overline{\delta })$ let $\mu _{\delta }$ be
an invariant Borel probability measure for the system $(X,T_{\delta })$; we
aim to get information on the regularity of this family of measures, by
investigating the regularity of the map $\delta \longmapsto \mu _{\delta }$.
This notion of regularity might depend on the topology with which the space
of measures is equipped. In this paper we will be interested in absolutely
continuous measures with the $L^{1}$ norm, as well as in the whole space of
Borel probability measures ${\mathcal{P}}(X)$, endowed with a suitable weak norm,
see subection \ref{sec1.1} for more details.

\medskip

We say that $(X,T_{0},\mu _{0})$ is \emph{statistically stable} (with
respect to the considered class of perturbations) if this map is continuous
at $\delta =0$ (with respect to the chosen topology on the space of measures
in which $\mu_0$ is perturbed). \emph{Quantitative statistical stability} is
provided by quantitative estimates on its modulus of continuity.

\smallskip

Differentiability of this map at $\delta =0$ is referred to by saying that
the system has \emph{linear response } to a certain class of perturbations.
Similarly, higher derivatives and higher degrees of smootheness can be
considered.

\medskip

These questions are by now well understood in the case of uniformly
hyperbolic systems, where it has been established Lipschitz and, in some
cases, differentiable dependence of the relevant (physical) invariant
measures with respect to the considered perturbation (see, for example, \cite%
{BB} for a recent survey on linear response under deterministic
perturbations, or the introduction in \cite{GS} for a survey focused on
higher-order terms in the response and for results in the stochastic
setting).

\smallskip

For systems having not a uniformly hyperbolic behavior, in presence of
discontinuities, or more complicated perturbations, much less is known and
results are limited to particular classes of systems; see, for instance, 
\cite{ASsu} for a general survey and \cite{A}, \cite{AV} ,\cite{BV}, \cite%
{BT}, \cite{BBS}, \cite{BKL}, \cite{BK2}, \cite{BS}, \cite{BS2}, \cite{BS2}, \cite{Dol}, \cite{D2}, 
\cite{D3}, \cite{GL}, \cite{Gmann}, \cite{Gpre}, \cite{Ko}, \cite{KL}, \cite%
{met}, \cite{Lin}, \cite{LS}, \cite{SV}, \cite{zz} for other results %
{about statistical stability} for different classes of
systems.
We point out a particular kind of deterministic perturbation which will be considered in this paper: the spatial discretization. In this perturbation, one considers a discrete set in the phase space and replaces the map $T$ with its composition with a projection to this discrete set.
This is what happens for example when we simulate the behavior of a system by iterating a map on our computer, which has a finite resolution and each iterate is subjected to numerical truncation. This perturbation changes the system into a periodic one, destroying many features of the original dynamics, yet this kind of simulations are quite reliable in many cases when the resolution is large enough and are widely used in the applied sciences. Why and under which assumptions these simulations are reliable or not is an important mathematical problem, which is still largely unsolved. Few
rigorous results have been found so far about the stability under spatial dicretization (see e.g. \cite{Bo},  \cite{GB}, \cite{Gu}, \cite{Gu2}, \cite{mier}). We refer to Section \ref{sectrunc} for a more detailed discussion on the subject.

\smallskip

The majority of results on statistical stability are established for systems
that are, in some sense, \textit{chaotic}. There is indeed a general
relation between the speed of convergence to the equilibrium of a system
(which reflects the speed of \textit{mixing}) and the quantitative aspects of its statistical stability
(see \cite{Gpre}, Theorem 5).

\medskip

In this paper we consider a class of systems that are not chaotic at all,
namely the {\emph{diffeomorphisms of the circle}}. We believe that they provide 
a good  model to start pushing forward this analysis.
{In particular, we will start our discussion by investigating the case of {\it rotations of the circle}, and then explaining how to generalize the results to the case of circle diffeomorphisms (see section \ref{sec:stabdiff}).}

\medskip

We prove the following results.

\begin{enumerate}
\item The statistical stability of irrational rotations under perturbations
that are small in the uniform convergence topology. Here stability is proved
with respect to a weak norm on the space ${\mathcal{P}}(X)$, related to the
so-called Wassertein distance; see Theorem \ref{statstab}.

\item H\"{o}lder statistical stability for Diophantine rotations under the
same kind of perturbations, where the H\"{o}lder exponent depends on the
Diophantine type of the rotation number. See Theorem \ref{stst2} for the
general upper bounds\ and Proposition \ref{berlusconi} for examples showing
these bounds are in some sense sharp.

\item Differentiable behavior and linear response for Diophantine rotations,
under smooth perturbations that preserve the rotation number; for general
smooth perturbations the result still holds, but for a Cantor set of
parameters (differentiability in the sense of Whitney); see Theorem \ref{KAMandResp} and Corollary \ref{corKAM}.

\item We extend these qualitative and quantitative stability results to diffeomorphisms of the circle {satisfying suitable assumptions}; see Theorems \ref{stadiff} and \ref{quantdiff}.

\item We prove the statistical stability of  diffeomorphisms of the circle 
under spatial discretizations and numerical truncations, also providing quantitative estimates on the ''error'' introduced by the discretization.

\end{enumerate}

{We believe that the general statistical stability picture here described for rotations is analogous to the one found, in different settings, for example in  \cite{BBS, BS2, BS1, LS} (see also \cite[Section 4]{BB}), where one has a smooth behavior for the response of statistical properties of the system to perturbations not changing the topological class of the system ({\it i.e.}, changing the system to a topologically conjugated one), while we have less regularity, and in particular H\"{o}lder behavior, if the perturbation is allowed to change it. In our case, the rotation number plays the role of determining the topological class of the system.}

Some comments on the methodology used to establish these results. As far as
items 1 and 2 are concerned, we remark that since rotations are not mixing,
the general relation between the speed of convergence to the equilibrium and
their statistical stability, that we have recalled above, cannot be applied.
However, we can perform some analogous construction considering the speed of
convergence to the equilibrium of the Ces\`{a}ro averages of the iterates of
a given measure, which leads to a measure of the speed of convergence of the
system to its ergodic behavior (see Lemma \ref{stablemma}). Quantitative
estimates of this speed of the convergence -- and hence our quantitative
stability statement, Theorem \ref{stst2} -- are obtained by means of the
so-called Denjoy-Koksma inequality (see Theorem \ref{DK}).

\smallskip 

On the other hand, results in item 3 are obtained as an application of KAM
theory for circle maps (see Theorem \ref{KAMVano}), with a particular focus
on the dependence of the KAM-construction on the perturbative parameter. In
Section \ref{KAMsection} we provide a brief introduction on this subject.%
\newline

The extension of the statistical stability results established for rotations
to{ circle diffeomorphisms} (item 4) is done again by
{combining our results for irrational rotations with the general theory of linearization of circle diffeomorphims, including Denjoy theorem, KAM theory and Herman-Yoccoz general theory (see section \ref{secconj})}.

The final application to spatial discretizations is obtained as  corollary of these statements, which -- thanks to the rather weak assumptions on the perturbations -- are suitable to deal with this particularly
difficult kind of setting.\\

{As a final remark, although we have decided to present our results in the framework of circle diffeomorphisms and rotations of the circle, we believe that the main ideas present in our constructions can be naturally applied to extend these results to rotations on higher dimensional tori. }\\

\noindent \textbf{Organization of the article.}
In Section \ref{sec1} after introducing some  tools from number theory and geometric measure theory we prove  qualitative and quantitative  statistical stability  of irrational rotations. The quantitative stability results are proved first by establishing general H\"older upper bounds in subsection \ref{ub} and then exhibiting particular small perturbations for which we actually have H\"older behavior, hence establishing lower bounds in section \ref{lob}. 

 In Section \ref{KAMsection}, after a brief introduction to  KAM theory and to the problem of smooth linearization of circle diffeomorphisms, we prove linear response results for suitable deterministic perturbations of Diophantine rotations.
 
  In Section \ref{sec:stabdiff} we show how to extend the results of Section \ref{sec1} to sufficiently smooth circle diffeomorphisms. 
  
  Finally, in Section \ref{sectrunc} we introduce a class of perturbations coming from spatial discretization and apply our previous results to this kind of perturbations, obtaining some qualitative and
 quantitative results.  
\newline

\noindent \textbf{Acknowledgments.} The authors are grateful to A. Celletti, R. de la Llave, P-A Guiheneuf,  C. Liverani, M. Sevryuk for their helpful suggestions. The authors also thank R. Calleja, A. Alessandra and R. de la Llave 
 for sharing with them their results in \cite{CCdL}.\\
 S.G. and A.S. have been partially supported by the
research project  PRIN Project 2017S35EHN ``{\it Regular and stochastic behavior in
dynamical systems}'' of the Italian Ministry of Education and Research (MIUR). 
AS also acknowledges the support of the MIUR Department of Excellence grant  CUP E83C18000100006.
\newline

\bigskip

\section{Statistical stability of irrational rotations}

\label{sec1}

Irrational rotations on the circle preserve the Lebesgue measure $m$ on the
circle {${\mathbb{S}}^1:= {\mathbb{R}}/{\mathbb{Z}}$} and are well known for
being uniquely ergodic. It is easy to see that small perturbations of such
rotations may have singular invariant measures {(\textit{i.e.}, not
absolutely continuous with respect to $m$)}, even supported on a discrete
set (see examples in Section \ref{lob}). However, we will show that these
measures must be {close, in some suitable sense,} to $m$.

\subsection{Weak statistical stability of irrational rotations}

\label{sec1.1}

In this section, we aim to prove a statistical stability result for
irrational rotations in a weak sense; more specifically, {we show that by
effect of small natural perturbations, their invariant measures vary
continuously with respect to the so-called Wassertein distance.}
{This qualitative result might not be surprising for experts, however
  the construction that we apply also leads to  quantitative estimates on the
statistical stability, which will be presented in the next subsections.}

\medskip

{Let us first recall some useful notions that we are going to use in the
following}. Let $(X,d)$ be a compact metric space and let ${\mathcal{M}}(X)$
denote the set of signed {finite} Borel measures on $X$. 
If $g:X\longrightarrow \mathbb{R}$ is a Lipschitz function, we denote its
(best) Lipschitz constant by $\mathrm{Lip}(g)$, \textit{i.e.} 
\begin{equation*}
\displaystyle{\mathrm{Lip}(g):=\sup_{x,y\in X,x\neq y}\left\{ \dfrac{%
|g(x)-g(y)|}{d(x,y)}\right\} }.
\end{equation*}

\smallskip

\begin{definition}
\label{w} Given $\mu, \nu \in {\mathcal{M}}(X) $ we define the \textbf{%
Wasserstein-Monge-Kantorovich} distance between $\mu $ and $\nu $ by%
\begin{equation}
W(\mu ,\nu ):=\sup_{\mathrm{Lip}(g)\leq 1,{\mathcal{M}}g{\mathcal{M}}%
_{\infty }\leq 1}\left\vert \int_{\mathbb{S}^1} {g}d\mu -\int_{\mathbb{S}^1} 
{g}d\nu \right\vert .
\end{equation}
We denote%
\begin{equation*}
\|\mu \|_{W}:=W(0,\mu ),
\end{equation*}%
{where $0$ denotes the trivial measure identically equal to zero.} $\|\cdot
\|_{W}$ defines a norm on the vector space of signed measures defined on a
compact metric space.
\end{definition}

{We refer the reader, for example, to \cite{AGS} for a more systematic and detailed description of these topics.}%
\newline

\medskip

Let $T:X\rightarrow X$ be a Borel measurable map. Define the linear
functional 
\begin{equation*}
L_{T}:{\mathcal{M}}(X)\rightarrow {\mathcal{M}}(X)
\end{equation*}%
that to a measure $\mu \in {\mathcal{M}}(X)$ associates the new measure $%
L_{T}\mu$, satisfying $L_{T}\mu (A):=\mu (T^{-1}(A)) $ for every Borel set $%
A\subset X$; {$L_{T}$ will be called \textit{transfer operator} (observe
that $L_{T}\mu$ is also called the {push-forward of $\mu$ by $T$} and
denoted by $T_*\mu$)}. If follows easily from the definition, that invariant
measures correspond to fixed points of $L_{T}$, \textit{i.e.}, $L_{T}\mu
=\mu $. 

\medskip

We are now ready to state our first statistical stability result for
irrational rotations.

\begin{theorem}[Weak statistical stability of irrational rotations.]
\label{statstab} Let $R_{\alpha }:{{\mathbb{S}}^{1}\rightarrow {\mathbb{S}}%
^{1}}$ be an irrational rotation. Let $\{T_{\delta }\}_{0\leq \delta \leq 
\overline{\delta }}$ be a family of Borel probability measurable maps of ${\mathbb{S}}^1$
to itself such that%
\begin{equation*}
\sup_{x\in {\mathbb{S}}^{1}}|R_{\alpha }(x)-T_{\delta }(x)|\leq \delta .
\end{equation*}%
Suppose $\mu _{\delta }$ is an invariant measure\footnote{%
In the case when $T_{\delta }$ is continuous such measures must exist by the
Krylov-Bogoliubov theorem { \cite{KB}}. In other cases such
measures can be absent, in this case our statement is empty.} of $T_{\delta
} $. Then 
\begin{equation*}
\lim_{\delta \rightarrow 0}\| m-\mu _{\delta }\|_{W}=0.
\end{equation*}
\end{theorem}

\bigskip

{Let us start with the following preliminary computation.}

\begin{lemma}
\label{stablemma}Let $L$ be the transfer operator associated to an isometry
of $\ {\mathbb{S}}^{1}$ and let $L_{\delta }$ be the transfer operator
associated to a measurable map $T_{\delta }$. Suppose that $\mu _{\delta
}=L_{\delta }\mu _{\delta }.$ Then, for each $n\geq 1$ 
\begin{equation}
\|\mu _{\delta }-m\|_{W} \;\leq\; \big\| m-\frac{1}{n}\sum_{{1\leq }i\leq
n}L^{i}\mu _{\delta } \big\|_{W} \;+\; \frac{(n-1)}{2} \; \big\|(L-L_{\delta
})\mu _{\delta } \big\|_{W}
\end{equation}%
where {$L^i := L \circ \ldots \circ L$ ($i$-times)}.
\end{lemma}

\medskip

\begin{proof}
The proof is a direct computation. Since $\mu _{\delta }=L_{\delta }\mu
_{\delta }$ and $m$ \ is invariant for $L$, then 
\begin{eqnarray}  \label{prodi}
\|\mu _{\delta }-m\|_{W} &\leq & \big \| \frac{1}{n}\sum_{1\leq i\leq
n}L_{\delta }^{i}\mu _{\delta }-\frac{1}{n}\sum_{1\leq i\leq n}L^{i}m \big\|%
_{W}  \notag \\
&\leq & \big\|\frac{1}{n}\sum_{1\leq i\leq n}L^{i}(m-\mu _{\delta }) \big\|%
_{W}+\big\|\frac{1}{n}\sum_{1\leq i\leq n}(L^{i}-L_{\delta }^{i})\mu
_{\delta }\big\|_{W}.
\end{eqnarray}%
Since 
\begin{equation*}
L^{i}-L_{\delta }^{i}=\sum_{k=1}^{i}L^{i-k}(L-L_{\delta })L_{\delta }^{k-1}
\end{equation*}%
then%
\begin{eqnarray*}
(L^{i}-L_{\delta }^{i})\mu _{\delta } &=&\sum_{k=1}^{i}L^{i-k}(L-L_{\delta
})L_{\delta }^{k-1}\mu _{\delta } \\
&=&\sum_{k=1}^{i}L^{i-k}(L-L_{\delta })\mu _{\delta }.
\end{eqnarray*}

Being $L$ is the transfer operator associated to an isometry, then 
\begin{equation}\label{mis}
\|L^{i-k}(L-L_{\delta })\mu _{\delta }\|_{W}\leq \|(L-L_{\delta })\mu _{\delta
}\|_{W}
\end{equation}
and consequently 
\begin{equation*}
{\Vert }(L^{i}-L_{\delta }^{i})\mu _{\delta }{\Vert _{W}}\leq
(i-1)\|(L-L_{\delta })\mu _{\delta }\|_{W}.
\end{equation*}
Substituting in \eqref{prodi}, we conclude 
\begin{equation*}
\|\mu _{\delta }-m\|_{W}\leq \big\|\frac{1}{n}\sum_{1\leq i\leq
n}L^{i}(m-\mu _{\delta })\big\|_{W}+\frac{(n-1)}{2}\|(L-L_{\delta })\mu
_{\delta }\|_{W}.
\end{equation*}
\end{proof}

\bigskip

\begin{lemma}
\label{prv1}Under the assumptions of Theorem \ref{statstab}, let $\{\mu
_{\delta }\}_{0\leq \delta \leq \overline{\delta }}$ be a family of Borel {probability}
measures on $\mathbb{S}^{1},$ then 
\begin{equation*}
\lim_{n\rightarrow \infty } \big \|m-\frac{1}{n}\sum_{1\leq i\leq n}L^{i}\mu
_{\delta } \big\|_{W}=0
\end{equation*}%
uniformly in $\delta$; namely, for every $\varepsilon>0$ there exists $%
\overline{n} = \overline{n}(\varepsilon)$ such that if $n\geq \overline{n}$
then 
\begin{equation*}
\sup_{0 \leq \delta \leq \overline{\delta}} \big\| m-\frac{1}{n}\sum_{1\leq
i\leq n}L^{i}\mu _{\delta } \big\|_{W} \leq \varepsilon.
\end{equation*}
\end{lemma}

\medskip

\begin{proof}
Let $\delta _{x_{o}}$ be the delta-measure concentrated at a point $x_{0}\in 
\mathbb{S}^{1}$. By unique ergodicity of the system, we get $%
\lim_{n\rightarrow \infty }\|m-\frac{1}{n}\sum_{1\leq i \leq n}L^{i}\delta
_{x_{0}}\|_{W}=0.$ This is uniform in $x_{0}$; in fact, changing $x_{0}$ is
equivalent to compose by a further rotation, which is an isometry and hence
does not change the $\|\cdot\|_{W}$ norm. Any measure $\mu _{\delta }$ can
be approximated in the $\|\cdot\|_{W}$ norm, with arbitrary precision, by a
convex combination of delta-measures, \textit{i.e.}, for each $\varepsilon
>0 $ there are $x_{1},...,x_{k}\in {\mathbb{S}}^1$and $\lambda
_{1},...,\lambda _{k}\geq 0$, with $\sum_{i\leq k}\lambda _{i}=1$ \ such
that 
\begin{equation*}
\big \|\mu _{\delta }-\sum_{1\leq i\leq k}\lambda _{i}\delta_{x_{i}} \big\|_{W}\leq
\varepsilon .
\end{equation*}

Since $R_{\alpha }$ is an isometry the $\|\cdot\|_{W}$ norm is preserved by
the iterates of $L.$ Hence for each $n\geq 0,$ we also have%
\begin{equation*}
\big \|L^{n}\mu _{\delta }-L^{n}\big(\sum_{1\leq i\leq k}\lambda _{i}\delta
_{x_{i}}\big) \big\|_{W}\leq \varepsilon,
\end{equation*}
which implies
\begin{equation*}
\big \|   m- L^{n}\mu _{\delta }\big\|_{W} \leq \varepsilon+ \big\| m -L^{n}\big(\sum_{1\leq i\leq k}\lambda _{i}\delta
_{x_{i}}\big) \big\|_{W}
\end{equation*}
and 

\begin{equation*}
\big \|   m- \frac{1}{n}\sum_{1\leq i\leq n}L^{i}\mu _{\delta } \big\|_{W} \leq \varepsilon+ \big\| m -\frac{1}{n}\sum_{1\leq j\leq n}L^{j}\big(\sum_{i\leq k}\lambda
_{i}\delta _{x_{i}}\big)\big\|_{W}.
\end{equation*}

We estimate now the behavior of the right hand side of the last inequality as $n\to \infty$. For any $n$ we have 
\begin{equation*}
\big \|m-\frac{1}{n}\sum_{1\leq j\leq n}L^{j}\big(\sum_{i\leq k}\lambda
_{i}\delta _{x_{i}}\big) \big\|_{W}= \big \|\sum_{1\leq i\leq k}\lambda
_{i}m-\sum_{1\leq i\leq k} \frac{\lambda_i}{n} \big( \sum_{1\leq j\leq
n}L^{j}\delta _{x_{i}} \big) \big\|_{W}
\end{equation*}%
\ and therefore $\lim_{n\rightarrow \infty }\|\sum_{i\leq k}\lambda _{i}\big(%
m-\frac{1}{n}\sum_{j\leq n}L^{j}\delta _{x_{i}}\big)\|_{W}=0 $. From this, the claim of the lemma easily follows.
\end{proof}

\bigskip

We can now prove Theorem \ref{statstab}.\newline

\begin{proof}[Proof of Theorem \protect\ref{statstab}]
Let $L_{\delta }$ be the transfer operator associated to $T_{\delta }.$ By
Lemma \ref{prv1}, $\lim_{n\rightarrow \infty }\|m-\frac{1}{n}\sum_{1\leq i
\leq n}L^{i}\mu _{\delta }\|_{W}=0$ uniformly in $\delta$. Since 
\begin{equation*}
\sup_{x\in {\mathbb{S}}^{1}}|R_{\alpha }(x)-T_{\delta }(x)|\leq \delta,
\end{equation*}
then $\|(L-L_{\delta })\mu _{\delta }\|_{W}\leq \delta $ and 
\begin{equation}  \label{covid}
\lim_{\delta \rightarrow 0}\|(L-L_{\delta })\mu _{\delta }\|_{W}=0.
\end{equation}
By Lemma \ref{stablemma} \ we get that for each $n$ 
\begin{equation}
\big \|\mu _{\delta }-m\|_{W}\leq \|m-\frac{1}{n}\sum_{1\leq i \leq
n}L^{i}\mu _{\delta } \big \|_{W}+\frac{(n-1)}{2} \big \|(L-L_{\delta })\mu
_{\delta }\big\|_{W}.
\end{equation}%
It follows from Lemma \ref{prv1} that we can choose $n$ such that $\|m-\frac{%
1}{n}\sum_{1\leq i \leq n}L^{i}\mu _{\delta }\|_{W}$ is as small as wanted.
Then, using \eqref{covid}, we can choose $\delta $ sufficiently mall so to
make $\frac{(n-1)}{2}\|(L-L_{\delta })\mu _{\delta }\|_{W}$ as small as
needed, hence proving the statement.
\end{proof}

\begin{remark}
The qualitative stability statements with respect to the Wasserstein distance proved in this section for circle rotations, extend directly to many other systems, for example to uniquely ergodic rotations on the multidimensional torus. In fact in the proof, aside of the general properies of the Wasserten distance and of pushforward maps, we only use that the system is uniquely ergodic, and the map is an isometry. This property could also be relaxed to a non-expansive property, ensuring that \eqref{mis}  is satisfied.
\end{remark}
\subsection{Quantitative statistical stability of Diophantine rotations,
upper bounds\label{ub}}

We now consider irrational rotations, {for rotation numbers that are
``badly'' approximable by rationals: the so-called \textit{Diophantine numers%
}. In this case, we can provide a quantitative estimate for the statistical
stability of the system by showing that the modulus of continuity of the function $\delta \longmapsto \mu_\delta$ is
H\"olderian, and that its exponent depends on the Diophantine type of the
rotation number.}

Let us start by recalling the definition of \textit{Diophantine type} for a
real number (see \cite{KN}): {this concept expresses quantitatively the rate
of approximability of an irrational number by sequences of rationals}. 
\newline
In what follows, we will also use {$\| \cdot \|_{\mathbb{Z}}$} to denote the
distance from a real number to the nearest integer.

\begin{definition}
\label{linapp} If $\alpha $ is irrational, the Diophantine type of $\alpha $
is defined by%
\begin{equation*}
\gamma (\alpha ):=\sup \{\gamma\geq 0: \underset{k\rightarrow \infty }{\lim
\inf }~\,k^{\gamma }\Vert k\alpha \Vert_{\mathbb{Z}} =0\mathbb{\}}.
\end{equation*}
\end{definition}

We remark that in some cases $\gamma (\alpha )=+\infty$. When $\gamma
(\alpha )<+\infty$ we say $\alpha $ is of \textit{finite Diophantine type}.%
\newline

\begin{remark}
The Diophantine type of $\alpha $ can be also defined by%
\begin{eqnarray*}
\gamma (\alpha )&:=&\inf \left\{ \gamma\geq 0: \,\exists c>0 \; \mbox{s.t.}
\;\Vert k\alpha \Vert_{\mathbb{Z}} \geq c_{0}|k|^{-\gamma } \; \forall \,
k\in \mathbb{Z}\setminus\{0\} \right\} \\
&=& \inf \left\{ \gamma\geq 0: \,\exists c>0 \; \mbox{s.t.} \; \big|\alpha -%
\frac{p}{q}\big|\geq \frac{c}{|q|^{\gamma +1}} \quad \forall \; \frac{p}{q}%
\in \mathbb{Q}\setminus\{0\} \right\}.
\end{eqnarray*}
\end{remark}

\medskip

{In the light of this last remark on the Diophantine type of a number, we
recall the definition of \textit{Diophantine number} as it very commonly
stated in the literature.}\newline

\begin{definition}
\label{DDD}Given $c >0$ and $\tau \geq 0$, we say that a number $\alpha \in
(0,1)$ is $(c ,\tau )$-\textit{Diophantine} if 
\begin{equation}
\left\vert \alpha -\frac{p}{q}\right\vert >\frac{c }{|q|^{1+\tau }}\qquad
\forall \quad \frac{p}{q}\in \mathbb{Q}\setminus \{0\}.  \label{diophantine}
\\
\end{equation}
We denote by $\mathcal{D}(c, \tau)$ the set of of $(c,\tau)$-{Diophantine}
numbers and by $\mathcal{D}(\tau) := \cup_{c>0} \mathcal{D}(c, \tau).$
\end{definition}

\medskip

\begin{remark}
{Comparing with Definition \ref{linapp}, it follows that every $\alpha \in 
\mathcal{D}(\tau)$ has finite Diophantine type $\gamma(\alpha)\leq \tau$. On
the other hand, if $\alpha$ has finite Diophantine type, then $\alpha \in 
\mathcal{D}(\tau )$ for every $\tau >\gamma (\alpha )$.}
\end{remark}

\begin{remark}
Let us point out the following properties (see \cite[p. 601]{Russmann} for
their proofs):

\begin{itemize}
\item \textrm{if $\tau<1$, the set $\mathcal{D}(\tau)$ is empty; }

\item \textrm{if $\tau>1$ the set $\mathcal{D}(\tau)$ has full Lebesgue
measure; }

\item \textrm{if $\tau=1$, then $\mathcal{D}(\tau)$ has Lebesgue measure
equal to ero, but it has Hausdorff dimension equal to $1$ (hence, it has the
cardinality of the continuum). }
\end{itemize}

\textrm{See also \cite[Section V.6]{Herman} for more properties.\newline
}
\end{remark}

Now we introduce the notion of discrepancy of a sequence $x_{1},...,x_{N}\in
\lbrack 0,1]$. This is a measure of the equidistribution of the points $%
x_{1},...,x_{N}$. Given $x_{1},...,x_{N}\in \lbrack 0,1]$ we define the
discrepancy of the sequence by%
\begin{equation*}
D_{N}(x_{1},...,x_{N}):=\sup_{\alpha \leq \beta ,~\alpha ,\beta \in \lbrack
0,1]}\big |\frac{1}{{N}}\sum_{1\leq i\leq N}1_{[\alpha ,\beta
]}(x_{i})-(\beta -\alpha )\big|
\end{equation*}%
it can be proved (see \cite[Theorem 3.2, page 123]{KN}) that the discrepancy
of sequences obtained from orbits of and irrational rotation is related to
the Diophantine type of the rotation number.

\begin{theorem}
\label{11}Let $\alpha $ be an irrational of finite Diophantine type. Let us
denote by $D_{N,\alpha }(0)$\ the discrepancy of the sequence $%
\{x_{i}\}_{0\leq i\leq N}=\{\alpha i-\left\lfloor \alpha i\right\rfloor
\}_{0\leq i\leq N}$ \ \ (where $\left\lfloor {\cdot}\right\rfloor $ \ stands
for the integer part). \ Then: 
\begin{equation*}
D_{N,\alpha }(0)=O(N^{-\frac{1}{\gamma (\alpha )}+\varepsilon }) \qquad
\forall\; \varepsilon>0.
\end{equation*}
\end{theorem}

\bigskip

From the definition of discrepancy, Theorem \ref{11}, and the fact that the
translation is an isometry, we can deduce the following corollary.\newline

\begin{corollary}
\label{preDK}Let $x_{0}\in [0,1]$, let us denote by $D_{N,\alpha
}(x_{0})$\ the discrepancy of the sequence $\{x_{i}\}_{1\leq i\leq
N}=\{x_{0}+\alpha i-\left\lfloor x_{0}+\alpha i\right\rfloor \}_{0\leq i\leq
N}$. Then Theorem \ref{11} holds uniformly for each $x_{0}$, namely for
every $\varepsilon >0$ there exists $C=C(\varepsilon )\geq 0$ \ such that
for each \thinspace $x_{0}$ and $N\geq 1$%
\begin{equation*}
D_{N,\alpha }(x_{0})\leq CN^{-\frac{1}{\gamma (\alpha )}+\varepsilon }.
\end{equation*}
\end{corollary}

\begin{proof}
It is sufficient to prove  that for each $x_0$ it holds that $D_{N,\alpha}(x_{0})\leq 2D_{N,\alpha}({0})$. Indeed, consider $\varepsilon>0 $ and an interval $I=[\alpha,\beta]$ such that
\begin{equation*}
D_{N}(x_{1},...,x_{N})-\varepsilon\leq \left| \frac{1}{{N}}\sum_{1\leq i\leq N}1_{I}(x_{i})-(\beta -\alpha )\right|.
\end{equation*}
Now consider the translation of $I$ by $-x_0$ (mod. $1$): $$S=\{x\in[0,1] \ | \  x+x_0-\lfloor x+x_0 \rfloor\in I\}$$ and the translation of the sequence $x_i$, which is the sequence $y_i=\alpha i-\left\lfloor \alpha i\right\rfloor
$. We have that $S $ is composed  by at most two intervals  $S=I_1\cup I_2$ with lenghts $m(I_1)$ and $m(I_2)$; moreover 
\begin{equation*}
\left |\frac{1}{{N}}\sum_{1\leq i\leq N}1_{I}(x_{i})-(\beta -\alpha )\right|= \left |\frac{1}{{N}}\sum_{1\leq i\leq N}1_{I_1}(y_{i})- m(I_1)+   \frac{1}{{N}}\sum_{1\leq i\leq N}1_{I_2}(y_{i})- m(I_2)     \right|.
\end{equation*}

Then

\begin{equation*}
D_{N}(x_{1},...,x_{N})-\varepsilon\leq 2 D_{N}(y_{1},...,y_{N}).
\end{equation*} 
Since $\varepsilon$ is arbitrary, we conclude that $D_{N,\alpha}(x_{0})\leq 2D_{N,\alpha}({0})$.
\end{proof}
\medskip

The discrepancy is also related to the speed of convergence of Birkhoff sums
of irrational rotations. The following is known as the Denjoy-Kocsma
inequality (see \cite[Theorem 5.1, page 143 and Theorem 1.3, page 91]{KN}).%
\newline

\begin{theorem}
\label{DK}Let $f$ be a function of bounded variation, that we denote by $%
V(f) $. Let $x_{1},...,x_{N}\in \lbrack 0,1]$ be a sequence with discrepancy 
$D_{N}(x_{1},...,x_{N})$. Then%
\begin{equation*}
\left|\frac{1}{N}\sum_{1\leq i\leq N}f(x_{i})-\int_{[0,1]}f~dx \right|\leq
V(f)\,D_{N}(x_{1},...,x_{N}).
\end{equation*}
\end{theorem}

\medskip

We can now prove a quantitative version of our stability result.\newline

\begin{theorem}[Quantitative statistical stability of Diophantine rotations]

\label{stst2} Let $R_{\alpha }:{{\mathbb{S}}^{1}\rightarrow {\mathbb{S}}^{1}}
$ be an irrational rotation. Suppose $\alpha $ has finite Diophantine type $%
\gamma (\alpha ).$ Let $\{T_{\delta }\}_{0\leq \delta \leq \overline{\delta }%
}$ be a family of Borel measurable maps of the circle such that%
\begin{equation*}
\sup_{x\in {\mathbb{S}}^{1}}|R_{\alpha }(x)-T_{\delta }(x)|\leq \delta .
\end{equation*}%
Suppose $\mu _{\delta }$ is an invariant measure of $T_{\delta }$. Then, for
each $\ell <{\frac{1}{\gamma (\alpha )+1}}$ we have: 
\begin{equation*}
\|m-\mu _{\delta }\|_{W}=O(\delta ^{\ell }).
\end{equation*}
\end{theorem}

\bigskip
Let us first prove some preliminary result.\newline

\begin{lemma}
\label{conv2}Under the assumptions of Theorem \ref{stst2}, let $\{\mu
_{\delta }\}_{0\leq \delta \leq \overline{\delta }}$ be a family of Borel {%
probability} measures on $\mathbb{S}^{1}$. Then, {for every $\varepsilon >0$}
\begin{equation}
\|m-\frac{1}{n}\sum_{1\leq i\leq n}L^{i}\mu _{\delta }\|_{W}=O(n^{-\frac{1}{%
\gamma (\alpha )}+\varepsilon })  \label{ww}
\end{equation}%
uniformly in $\delta$; namely, {for every $\varepsilon >0$}, there exist $C={%
C(\varepsilon )}\geq 0$ such that for each $\delta $ and $n\geq 1$ 
\begin{equation*}
\|m-\frac{1}{n}\sum_{1\leq i\leq n}L^{i}\mu _{\delta }\|_{W}\leq Cn^{-\frac{1%
}{\gamma (\alpha )}+\varepsilon }.
\end{equation*}
\end{lemma}

\bigskip

\begin{proof}
Let us fix $\varepsilon >0.$ By Theorem \ref{DK} and Corollary \ref{preDK}
we have that there is $C\geq 0$ such that for each Lipschitz function $f$
with Lipschitz constant $1$, and for each $x_{0}\in {\mathbb{S}}^{1}$ we have%
\begin{equation*}
\left|\frac{1}{n}\sum_{1\leq i\leq n}f(R_{\alpha
}^{i}(x_{0}))-\int_{[0,1]}f~dx\right |\leq C\, n^{-\frac{1}{\gamma (\alpha )}%
+\varepsilon } \qquad \forall\; n\geq 1.
\end{equation*}

Let $\delta _{x_{0}}$ be the delta-measure concentrated at a point $x_{0}\in 
\mathbb{S}^{1}$. By definition of $\|\cdot\|_{W}$, we conclude that 
\begin{equation}
\|m-\frac{1}{n}\sum_{1\leq i\leq n}L^{i}\delta _{x_{0}}\|_{W}\leq Cn^{-\frac{%
1}{\gamma (\alpha )}+\varepsilon }.  \label{www}
\end{equation}%
\newline
Now, as in the proof of Lemma \ref{stablemma}, any measure $\mu _{\delta }$
can be approximated, arbitary well, in the $\|\cdot\|_{W}$ norm by a convex
combination of delta-measures and we obtain $($\ref{ww}$)$ from $(\ref{www})$%
, exactly in the same way as done in the proof of Lemma \ref{stablemma}.
\end{proof}

\bigskip

\begin{proof}[Proof of Theorem \protect\ref{stst2}]
Let $L_{\delta }$ be the transfer operator of $T_{\delta }.$ Let us fix $%
\varepsilon >0$; without loss of generality we can suppose $\varepsilon <%
\frac{1}{\gamma (\alpha )}.$ By lemma \ref{conv2} we have that 
\begin{equation*}
\|m-\frac{1}{n}\sum_{1\leq i \leq n}L^{i}\mu _{\delta }\|_{W}\leq Cn^{-\frac{%
1}{\gamma (\alpha )}+\varepsilon }.
\end{equation*}
By Lemma \ref{stablemma} \ we get that for each $n\geq 1$ 
\begin{equation}
\|\mu _{\delta }-m\|_{W}\leq \big \|m-\frac{1}{n}\sum_{1\leq i \leq
n}L^{i}\mu _{\delta }\big\|_{W}+\frac{(n-1)}{2}\big \|(L-L_{\delta })\mu
_{\delta }\big\|_{W}.
\end{equation}

Hence%
\begin{eqnarray}
\|\mu _{\delta }-m\|_{W} &\leq &Cn^{-\frac{1}{\gamma (\alpha )}+\varepsilon
}+\frac{(n-1)}{2}\|(L-L_{\delta })\mu _{\delta }\|_{W}  \label{stimaboh} \\
&\leq &{\ Cn^{-\frac{1}{\gamma (\alpha )}+\varepsilon }+\frac{(n-1)}{2}%
\delta },  \notag
\end{eqnarray}%
{\ where we have used that, since }$\sup_{x\in {\mathbb{S}}^{1}}|R_{\alpha
}(x)-T_{\delta }(x)|\leq \delta $, then 
\begin{equation*}
\|(L-L_{\delta })\mu _{\delta }\|_{W}\leq \delta.
\end{equation*}

Since the inequality is true for each $n\geq1$, we can now consider $n$
minimizing 
\begin{equation*}
{F(n):=Cn^{-\frac{1}{\gamma (\alpha )}+\varepsilon }+\frac{n-1}{2}\delta .}
\end{equation*}%
The extension to $\mathbb{R}$ of the funcion $F$ is convex {and it goes to $%
+\infty $ both as $x\rightarrow 0^{+}$ and as $x\rightarrow +\infty $.} Let
us denote \ $a:=\frac{1}{\gamma (\alpha )}-\varepsilon {>0}$, then $%
F(x)=Cx^{-a}+\frac{x-1}{2}\delta .$ This is minimized at %
{\ 
\begin{equation*}
x_{\ast }:=(2aC)^{\frac{1}{a+1}}\delta ^{-{\frac{1}{a+1}}}:=\tilde{c}%
\;\delta ^{-\frac{1}{a+1}}.
\end{equation*}%
} {\ Consider $n_{\ast }=\left\lfloor x_{\ast }\right\rfloor $ and observe
that%
\begin{eqnarray*}
F(n_{\ast }) &=&\frac{C}{n_{\ast }^{a}}+\frac{n_{\ast }-1}{2}\delta \leq 
\frac{C}{n_{\ast }^{a}}+\frac{n_{\ast }}{2}\delta =O(\delta ^{\frac{a}{a+1}})
\\
F(n_{\ast }+1) &=&\frac{C}{(n_{\ast }+1)^{a}}+\frac{n_{\ast }}{2}\delta \leq 
\frac{C}{n_{\ast }^{a}}+\frac{n_{\ast }}{2}\delta =O(\delta ^{\frac{a}{a+1}%
}).
\end{eqnarray*}%
}

Substituting in \eqref{stimaboh} we conclude:%
\begin{eqnarray*}
\|\mu _{\delta }-m\|_{W} &\leq &\min \{F(n_{\ast }),F(n_{\ast
}+1)\}=O(\delta ^{\frac{a}{a+1}}) \\
&=&O\big(\delta ^{\frac{1-\varepsilon \gamma (\alpha )}{1+(1-\varepsilon
)\gamma (\alpha )}}\big)
\end{eqnarray*}%
proving the statement.
\end{proof}

\begin{remark}
{We remark that, as it follows from the above proof, the constants involved in  $O(\delta ^{\ell })$ in the statement of Theorem \ref{stst2} only depend on $\alpha $ and $\ell$.}
\end{remark}

\subsection{Quantitative statistical stability of Diophantine rotations,
lower bounds\label{lob}}

{In this subsection we discuss that the upper bound on the statistical
stability obtained in Theorem \ref{stst2} is essentially optimal.} We show
that for a rotation $R_{\alpha }$ with rotation number $\alpha$ of
Diophantine type $1< \gamma(\alpha) \leq +\infty$, there exist perturbations
of ``size $\delta$'', for which the unique physical invariant measure {varies} in a H%
\"{o}lder way.\newline
More specifically, for any $r\geq 0$ we will construct a sequence $\delta
_{n}\rightarrow 0$ \ and $C^\infty$-maps $T_{n}$ such that: $\|R_{\alpha
}-T_{n}\|_{C^{r}}\leq \delta _{n}$, $T_{n}$ has {a} unique physical invariant {%
probability} measure $\mu _{n}$ and $\|\mu _{n}-m \|_{W}\geq C\delta_n ^{\frac{%
1}{p}}$ for some $C\geq 0$ and $p>1$.\newline

\begin{proposition}
\label{berlusconi} Let us consider the rotation $R_{\alpha }:\mathbb{S}%
^{1}\rightarrow \mathbb{S}^{1}$, where $\alpha$ is an irrational number with 
{$1< \gamma(\alpha) \leq +\infty$}. For each {$r\geq 0$ } and $\gamma
^{\prime }<\mathcal{\gamma }(\alpha )$ there exist a sequence of numbers $%
\delta _{j}> 0 $ and $C^\infty$ diffeomorphisms $T_{j}:\mathbb{S}%
^{1}\rightarrow \mathbb{S}^{1}$ such that $\|T_{j}-R_{\alpha }\|_{C^{r}}\leq
2\delta _{j}$ \ and 
\begin{equation*}
\|m-\mu _{j}\|_{W}\geq \frac{1}{2}{\delta _{j}^{\frac{1}{\gamma ^{\prime
}+1}}}
\end{equation*}%
for every $j\in \mathbb{N}$ and for every $\mu _{j}$ invariant measure of $%
T_{j}$.
\end{proposition}

\begin{proof}
We remark that the unique invariant measure for $R_{\alpha }$ is the Lebesgue
measure $m.$ Let us choose $\gamma ^{\prime }<\gamma (\alpha )$; {it follows
from the definition of $\gamma (\alpha )$} that there are infinitely many
integers ${k_{j}\in \mathbb{N}}$ and ${p_{j}\in \mathbb{Z}}$ such that 
\begin{equation*}
|k_{j}\alpha -p_{j}|\leq \frac{1}{k_{j}^{\gamma ^{\prime }}}\qquad
\Longleftrightarrow \qquad \big|\alpha -\frac{p_{j}}{k_{j}}\big|\leq \frac{1%
}{k_{j}^{\gamma ^{\prime }+1}}.
\end{equation*}

Let us set $\delta _{j}:=-\alpha +\frac{p_{j}}{k_{j}}$. Clearly, $|\delta
_{j}|\leq \frac{1}{k_{j}^{\gamma ^{\prime }+1}}\longrightarrow 0$ as $%
j\rightarrow \infty $.

Consider $\hat{T}_{j}$ defined as $\hat{T}_{j}(x)=R_{\alpha +\delta _{j}}(x)$%
; for each $r\geq 0$ we have that \ $\|\hat{T}_{j}-R_{\alpha
}\|_{C^{r}}=|\delta _{j}|$. Since $(\delta _{j}+\alpha )=\frac{p_{j}}{k_{j}}
$ is rational, {every orbit is $k_{j}$-periodic. Let us consider the orbit
starting at $0$ and denote it by} 
\begin{equation*}
y_{0}:=0,\;y_{1}:=\delta _{j},\;\ldots ,\;y_{k_{j}-1}:=1-\delta
_{j},\;y_{k_{j}}:=0\;(\mathrm{mod.} \,{\mathbb{Z}}).
\end{equation*}%
Consider the measures 
\begin{equation*}
\mu _{j}=\frac{1}{k_{j}}\sum_{0\leq i<k_{j}}\delta _{y_{i}},
\end{equation*}%
where $\delta _{y_{i}}$ is the delta-measure concentrated at $y_{i}$. 
The measure $\mu _{j}$ is clearly invariant for the map $\hat{T}_{j}$ and
it can be directly computed that%
\begin{equation*}
\|m-\mu _{j}\|_{W}\geq \frac{1}{2k_{j}}.
\end{equation*}%

{Observe that $|\delta _{j}|\leq \frac{1}{k_{j}^{\gamma ^{\prime }+1}}$,
hence} we get $|\delta _{j}|^{\frac{1}{\gamma ^{\prime }+1}}\leq \frac{1}{%
k_{j}}$; then%
\begin{equation*}
\|m-\mu _{j}\|_{W}\geq \frac{1}{2}{|\delta _{j}|^{\frac{1}{\gamma
^{\prime }+1}}}.
\end{equation*}

This example can be further improved by perturbing the map $\hat{T}%
_{j}=R_{\alpha +\delta _{j}}$ to a new map $T_{j}$ in a way that the measure 
$\mu _{j}$ (supported on the attractor of $T_{j}$) and the measure \footnote{%
The \textit{translated measure} is defined as follows: $[\mu _{j}+\frac{1}{%
2k_{j}}](A):=\mu _{j}(A-\frac{1}{2k_{j}})$ \ for each measurable set $A$ in $%
\mathbb{S}^{1}$, where $A-\frac{1}{2k_{j}}$ is the translation of the set $A$
by $-\frac{1}{2k_{j}}$.} $\mu _{j}+\frac{k_{j}}{2}$ (supported on the
repeller of $T_{j}$) are the only invariant measures of $T_{j}$, and $\mu
_{j}$ is the unique physical measure for the system. This can be done by
making a $C^{\infty }$ perturbation on $\hat{T}_{j}=R_{\alpha +\delta _{j}}$%
, as small as wanted in the $C^{r}$-norm. In fact, let us denote, as before,
by $\{y_{k}\}_{k}$\ the periodic orbit of $0$ for $R_{\alpha +\delta _{j}}$.
Let us consider a $C^{\infty }$ function $g:[0,1]\rightarrow \lbrack 0,1]$
such that:

\begin{itemize}
\item $g$ is negative on the each interval $[y_{i},y_{i}+\frac{1}{2k_{j}}]$
and positive on each interval $[y_{i}+\frac{1}{2k_{j}},y_{i+1}]$ (so that $%
g(y_{i}+\frac{1}{2k_{j}})=0$ );

\item $g^{\prime }$ is positive in each interval $[y_{i}+\frac{1}{3k_{j}}%
,y_{i+1}-\frac{1}{3k_{j}}]$ and negative in $[y_{i},y_{i+1}]-[y_{i}+\frac{1}{%
3k_{j}},y_{i+1}-\frac{1}{3k_{j}}]$.
\end{itemize}

Considering $D_{\delta }:{\mathbb{S}}^{1}\rightarrow {\mathbb{S}}^{1}$,
defined by $D_{\delta }(x):=x+\delta g(x)$ $\func{(mod.\; {\mathbb{Z}})}$,
it holds that the iterates of this map send all the space, with the
exception of the set $\Gamma_{\mathrm{rep}}:=\{y_{i}+\frac{1}{2k_{j}}: \;0
\leq i< k_{j}\}$ (which is a repeller), to the set $\Gamma_{ \mathrm{att}%
}:=\{y_{i}: \; 0\leq i< k_{j}\}$ (the attractor). Then, define $\ T_{j}$ by
composing $R_{\alpha +\delta _{j}}$ and $D_{\delta }$, namely 
\begin{equation*}
T_{j}(x):=D_{\delta _{j}}(x+(\delta _{j}+\alpha )).
\end{equation*}

The claim follows by observing that for the map $T_{j}(x)$, both sets \ $%
\Gamma_{\mathrm{att}}$ and $\Gamma_{\mathrm{rep}}$ are invariant and, in
particular, the whole space ${\mathbb{S}}^{1}-\Gamma_{\mathrm{rep}}$ is
attracted by $\Gamma_{\mathrm{att}}$.
\end{proof}

\bigskip

The construction done in the previous proof can be extended to show H\"{o}%
lder behavior for the average of a given \emph{fixed} regular observable. We
show an explicit example of such an observable, with a particular choice of
rotation number $\alpha$.

\begin{proposition}
\label{30}Consider a rotation $R_{\alpha }$ with rotation angle $\alpha
:=\sum_{1}^{\infty }2^{-2^{2i}}$. Let $T_{j}$ be its perturbations as
constructed in Proposition \ref{berlusconi} and let $\mu _{j}$ denote their
invariant measures; recall that $\|T_{j}-R_{\alpha }\|_{C^{k}}\leq 2|\delta
_{j}|=2\sum_{n+1}^{\infty }2^{-2^{2i}}$.\newline
Then, there is an observable $\psi :{\mathbb{S}}^{1}\rightarrow \mathbb{R}$,
with derivative in $L^{2}({\mathbb{S}^1})$, and $C\geq 0$ such that%
\begin{equation*}
\left |\int_{\mathbb{S}^1} \psi d{m}-\int_{\mathbb{S}^1} \psi d\mu _{j}
\right|\geq C\sqrt{\delta _{j}}.
\end{equation*}
\end{proposition}

\bigskip

\begin{proof}
Comparing the series with a geometric one, we get that 
\begin{equation*}
\sum_{n+1}^{\infty }2^{-2^{2i}}\leq 2^{-2^{2(n+1)}+1}.
\end{equation*}
By this, it follows 
\begin{equation*}
\|2^{2^{2n}}\alpha \|\leq 2^{-2^{2(n+1)}+1}=\frac{1}{2(2^{2^{2+2n}})}=\frac{1%
}{2(2^{2^{2n}})^{4}}.
\end{equation*}
Since it also holds that $\|2^{2^{2n}}\alpha \|\geq 2^{-2^{2(n+1)}}$, the we
conclude that $\gamma (\alpha )=$ $4$. Following the construction in the
proof of Proposition \ref{berlusconi}, we have that with a perturbation of
size less than $2^{-2^{2(n+1)}+1}$ the angles $\alpha _{j}:=\alpha -\delta
_{j}=\sum_{1}^{j}2^{-2^{2i}}$ generate orbits of period $2^{2^{2j}}$. Now
let us construct a suitable observable which can ``see'' the change of the
invariant measure under this perturbation. Let us consider 
\begin{equation}
\psi (x):=\sum_{i=1}^{\infty }\frac{1}{(2^{2^{2i}})^{2}}\cos (2^{2^{2i}}2\pi
x)  \label{obss}
\end{equation}%
and debote by $\psi _{k}(x):=\sum_{i=1}^{k}\frac{1}{(2^{2^{2i}})^{2}}\cos
(2^{2^{2i}}2\pi x)$ its truncations. Since for the observable $\psi $, the $%
i $-th Fourier coefficient decreases like $i^{-2}$, then $\psi $ has
derivative in $L^{2}({\mathbb{S}^1})$. Let $\{x_{i}\}_{i}$ 
be the periodic orbit of $0$ for the map $R_{\alpha _{j}}$ 
and let $\mu _{j}:=\frac{1}{2^{2^{2i}}}\sum_{i=0}^{\alpha _{j}-1}\delta
_{x_{i}}$ be the physical measure supported on it. Since $2^{2^{2j}}$\
divides $2^{2^{2(j+1)}}$ then\ $\sum_{i=1}^{2^{2^{2j}}}\psi _{k}(x_{i})=0$
for every $k<j$, thus $\int_{\mathbb{S}^1} \psi _{j-1}~d\mu _{j}=0.$ Then 
\begin{eqnarray*}
v_{j}:= &&\int_{\mathbb{S}^1} \psi ~d\mu _{j}\geq \frac{1}{(2^{2^{2j}})^{2}}%
-\sum_{j+1}^{\infty }\frac{1}{(2^{2^{2i}})^{2}} \\
&\geq &2^{-2^{2j+1}}-2^{-2^{2(j+1)}+1}.
\end{eqnarray*}%
For $j$ big enough%
\begin{equation*}
2^{-2^{2j+1}}-2^{-2^{2(j+1)}+1}\geq \frac{1}{2}(2^{-2^{2j}})^{2}.
\end{equation*}%
Summarizing, with a perturbation of size 
\begin{equation*}
\delta _{j}=\sum_{j+1}^{\infty }2^{-2^{2i}}\leq 2\cdot
2^{-2^{2(j+1)}}=2^{-2^{2(j+1)}}=2(2^{-2^{2j}})^{4}
\end{equation*}
we get a change of average for the observable $\psi $ from $\int_{\mathbb{S}%
^1} \psi dm=0$ to $v_{n}\geq \frac{1}{2}(2^{-2^{2j}})^{2}$. Therefore, there
is $C\geq 0$ such that with a perturbation of size $\delta _{j}$, we get a
change of average for the observable $\psi $ of size bigger than $C\sqrt{%
\delta _{j}}.$
\end{proof}

\bigskip

\begin{remark}
Using in (\ref{obss}) {$\frac{1}{(2^{2^{2i}})^{\sigma}}$, for some $\sigma>2$}, instead of $\frac{1%
}{(2^{2^{2i}})^{2}}$, we can obtain a smoother observable. Using rotation
angles with bigger and bigger Diophantine type, it is possible to obtain a
dependence of the physical measure on the perturbation with worse and worse H%
\"{o}lder exponent. Using angles with infinite Diophantine type it is
possible to have a behavior whose modulus of continuity is worse than the H%
\"{o}lder one.
\end{remark}

\bigskip

\section{Linear response and KAM theory}

\label{KAMsection} {In this section, we would like to discuss differentiable
behavior and linear response for Diophantine rotations, under suitable
smooth perturbations. In particular, we will obtain our results by means of
the so-called KAM theory.}

{Let us first start by explaining more precisely,} what linear response
means.\newline
Let $(T_{\delta })_{\delta \geq 0}$ be a one parameter family of maps
obtained by perturbing an initial map $T_{0}$. We will be interested on how
the perturbation made on $T_{0}$ affects some invariant measure of $T_{0}$
of particular interest. For example its physical measure. Suppose hence $%
T_{0}$ has a physical measure $\mu _{0}$ and let $\mu _{\delta }$ be\
physical measures of $T_{\delta }$. \footnote{\label{notap} An invariant
measure $\mu $ is said to be \emph{physical} if there is a positive Lesbegue
measure set $B$ such that for each continuous observable $f $%
\begin{equation*}
\int_{\mathbb{S}^1} f~d\mu =\underset{n\rightarrow \infty }{\lim }\frac{%
f(x)+f(T(x))+...+f(T^{n}(x))}{n+1}
\end{equation*}%
for each $x\in B$ (see \cite{Y}).}

The linear response of the invariant measure of $T_{0}$ under {a} given
perturbation is defined, {if it exists}, by the limit 
\begin{equation}
\dot{\mu}:=\lim_{\delta \rightarrow 0}\frac{\mu _{\delta }-\mu _{0}}{\delta }
\label{LRidea}
\end{equation}
where the meaning of this convergence can vary from system to system. In
some systems and for a given perturbation, one may get $L^{1}$-convergence
for this limit; in other systems or for other perturbations one may get
convergence in weaker or stronger topologies. The linear response to the
perturbation hence represents the first order term of the response of a
system to a perturbation and when it holds, a linear response formula can be
written {as}: 
\begin{equation}
\mu _{\delta }=\mu _{0}+\dot{\mu}\delta +o(\delta )  \label{lin}
\end{equation}%
which holds in some weaker or stronger sense.

We remark that given an observable function $c:X\rightarrow \mathbb{R}$, if
the convergence in \eqref{LRidea} is strong enough with respect to the
regularity \footnote{%
For example, $L^{1}$ convergence in $($\ref{LRidea}$)$ allows to control the
behavior of $L^{\infty }$ observables in $($\ref{LRidea2}$)$, while a weaker
convergence in $($\ref{LRidea}$)$, for example in the Wasserstein norm (see
definition \ref{w})\ allows to get information on the behavior of Lipschitz
obsevable.} of $c$, we get

\begin{equation}
\lim_{t\rightarrow 0}\frac{\int_{\mathbb{S}^1} \ c\ d\mu _{t}-\int_{\mathbb{S%
}^1} \ c\ d\mu _{0}}{t}=\int_{\mathbb{S}^1} \ c\ d\dot{\mu}  \label{LRidea2}
\end{equation}%
showing how the linear response of the invariant measure controls the
behavior of observable averages.\newline

\subsection{Conjugacy theory for circle maps} \label{secconj}

{Let us recall some classical results on smooth linearization of circle
diffeomorphisms and introduce KAM theory}.

Let $\mathrm{Diff}_+^r({{\mathbb{S}}^1})$ denote the set of orientation
preserving homeomorphism of the circle of class $C^r$ with $r\in \mathbb{N}%
\cup \{+\infty, \omega \}$. Let $\mathrm{rot}(f) \in {{\mathbb{S}}^1}$
denote the rotation number of $f$ (see, for example, \cite[Section II.2]%
{Herman} for more properties on the rotation number).\newline

A natural question is to understand when a circle diffeomorphism is
conjugated to a rotation with the same rotation number, namely whether there
exists a homeomorphim $h: {\ \mathbb{S}^1 }\longrightarrow {{\mathbb{S}}^1}$
such that the following diagram commutes: 
\begin{equation*}
\begin{array}{ccc}
{{\mathbb{S}}^1} & \overset{f}{\longrightarrow } & {{\mathbb{S}}^1} \\ 
\uparrow {\small h} &  & \uparrow {\small h} \\ 
{{\mathbb{S}}^1} & \overset{R_{\mathrm{rot (f)}}}{\longrightarrow } & {{%
\mathbb{S}}^1}%
\end{array}%
\end{equation*}
\textit{i.e.}, $h^{-1} \circ f \circ h = R_{\mathrm{rot}(f)}$. Moreover,
whenever this conjugacy exists, one would like to understand what is the
best regularity that one could expect.

\begin{remark}
\textrm{Observe that if $h$ exists, then it is essentially unique, in the
sense that if $h_i: \mathbb{S}^1\longrightarrow {{\mathbb{S}}^1}$, $i=1,2$,
are homeomorphisms conjugating $f$ to $R_{\mathrm{rot }(f)}$, then $h_1
\circ h_2^{-1}$ must be a rotation itself: $h_1\circ h_2^{-1} = R_\beta$ for
some $\beta\in {\mathbb{S}}^1 $ (see \cite[Ch. II, Proposition 3.3.2]{Herman}%
).}
\end{remark}

This question has attracted a lot of attention, dating back, at least, to
Henri Poincar\'e. 

{Let us start by recalling the following result due to Denjoy  \cite{Den} shows that
diffeomorphisms  with irrational rotation number and satisfying some extra mild regularity assumption (for example, $C^2$ diffeomorphisms do satisfy it) are conjugated to irrational
rotations by an homeomorphism.}

\begin{theorem}[Denjoy]
\label{Denteo}Let $T$ be an orientation preserving  diffeomorphism of the
circle with an irrational rotation number $\alpha $ and such that $\log
(T^{\prime })$ has bounded variation. Then there exists a homeomorphism $h:%
\mathbb{S}^{1}\rightarrow \mathbb{S}^{1}$ such that%
\begin{equation*}
{T \circ h= h \circ R_{\alpha }.}
\end{equation*}
\end{theorem}

{
\begin{remark}
Denjoy constructed diffeomorphisms $T$ only of class $C^1$ that are not conjugated to rotations ({\it i.e.}, such that the support of their invariant measure $\mu$ is not the whole ${\mathbb S^1}$). These are usually called in the literature {\it Denjoy-type} diffeomorphisms.
\end{remark}
}

{Some of the first contributions about smooth linearization ({\it i.e.}, obtaining a conjugacy of higher regularity)} were due to V.I. Arnol'd 
\cite{Arnold} and J. Moser \cite{Moser}. These results are in the
perturbative setting and are generally referred to as \textit{KAM theory}.
Namely, they consider perturbations of \textit{Diophantine} rotations 
\begin{equation}  \label{deffeps}
f_\varepsilon(x) = R_\alpha + \varepsilon u(x,\varepsilon)
\end{equation}
and prove that, under suitable regularity assumptions on $u$, there exist $%
\varepsilon_0>0$ (depending on the properties of $\alpha$ and $u$) 
{and a Cantor set  ${\mathcal C} \subset (-\varepsilon_0, \varepsilon_0)$ such that $f_\varepsilon$ is conjugated to a $R_{\mathrm{rot} (f_\varepsilon)}$ for every $\varepsilon \in {\mathcal C}$.} 
{Observe that the conjugacy does not exist in general for an interval of $\varepsilon$, but only for those values of $\varepsilon$ for which the rotation number of $f_\varepsilon$ satisfies
suitable arithmetic properties ({\it e.g.}, it is Diophantine)}.
See below for a more precise statement.

\begin{remark}
{Observe that $f_\varepsilon$ has not necessarily rotation number $\alpha$,
even if one asks that $u(\cdot, \varepsilon)$ has zero average.}
\end{remark}

\begin{remark}
In the analytic setting, KAM theorem for circle diffeomorphisms was firstly
proved by Arnol'd (see \cite[Corollary to Theorem 3, p. 173]{Arnold}),
showing that the conjugation is analytic. In the smooth case, it was proved
by Moser \cite{Moser} under the assumption that $u$ is sufficiently smooth
(the minimal regularity needed was later improved by R\"ussmann \cite%
{Russmann2}). The literature on KAM theory and its recent developments is
huge and we do not aim to provide an accurate account here; for reader's
sake, we limit ourselves to mentioning some recent articles and surveys,
like \cite{BroerSevryuk, DL, Dumas, Massetti, MatherForni, Wayne} and
references therein.
\end{remark}

\smallskip

Later, Herman \cite{Herman} and Yoccoz \cite{Yoccoz,Yoccoz2} provided a
thorough analysis of the situation in the general (non-perturbative)
context. Let us briefly summarize their results (see also \cite%
{EliassonFayadKrikorian} for a more complete account). \newline

\begin{theorem}[Herman \protect\cite{Herman}, Yoccoz \protect\cite{Yoccoz,
Yoccoz2}]
\textcolor{white}{per andare a capo} \label{thmhermanyoccoz}

\begin{itemize}
\item Let $f \in \mathrm{Diff}_+^r({{\mathbb{S}}^1})$ and $\mathrm{rot}(f)
\in \mathcal{D}(\tau)$. If $r>\max\{3, 2\tau-1\}$, then there exists $h\in 
\mathrm{Diff}_+^{r-\tau-\varepsilon}({{\mathbb{S}}^1})$, for every $%
\varepsilon>0 $, conjugating $f$ to $R_{\mathrm{rot (f)}}$.

\item Let $f \in \mathrm{Diff}_+^\infty({{\mathbb{S}}^1})$ and $\mathrm{rot}%
(f) \in \mathcal{D}(\tau)$. Then, there exists $h\in \mathrm{Diff}%
_+^{\infty}({{\mathbb{S}}^1})$ conjugating $f$ to $R_{\mathrm{rot (f)}}$.

\item Let $f \in \mathrm{Diff}_+^\omega({{\mathbb{S}}^1})$ and $\mathrm{rot}%
(f) \in \mathcal{D}(\tau)$. Then, there exists $h\in \mathrm{Diff}%
_+^{\omega}({{\mathbb{S}}^1})$ conjugating $f$ to $R_{\mathrm{rot (f)}}$.%
\newline
\end{itemize}
\end{theorem}

\begin{remark}
\textrm{The above results can be generalized to larger classes of rotation
number, satisfying a weaker condition than being Diophantine. Optimal
conditions were studied by Yoccoz and identified in \textit{Brjuno numbers}
for the smooth case and in those satisfying the so-called ${\mathcal{H}}$%
-condition (named in honour of Herman); we refer to \cite{Yoccoz, Yoccoz2}
for more details on these classes of numbers.\newline
}
\end{remark}

\bigskip

\subsection{Linear response for Diophantine circle rotations}

In this subsection we describe how, as a corollary to KAM theory, one can
prove the existence of linear response for Diophantine rotations.\newline

Let us state the following version of KAM theorem, whose proof can be found
in \cite[Theorem 9.0.4]{Vano} (cf. also \cite[Theorem 2]{BroerSevryuk} and \cite{CCdL}). %

\medskip

\begin{theorem}[KAM Theorem for circle diffeomorphisms]
\label{KAMVano} Let $\alpha \in \mathcal{D}({\tau})$, with $\tau>1$ and let
us consider a smooth family of circle diffeomorphisms 
\begin{equation*}
f_\varepsilon(x) = R_\alpha + \varepsilon u(x,\varepsilon) \qquad
|\varepsilon|< 1
\end{equation*}
with

\begin{itemize}
\item[\textrm{(i)}] $u(x,\varepsilon) \in C^{\infty}({\mathbb{S}}^1)$ for
every $|\varepsilon|<1$;

\item[\textrm{(ii)}] the map $\varepsilon \longmapsto u(\cdot, \varepsilon)$
is $C^{\infty}$;

\item[\textrm{(iii)}] $\int_{{\mathbb{S}}^1} u(x,\varepsilon) dx =
A\varepsilon^m + o(\varepsilon^m)$, where $A\neq 0$ and $m\geq 0$.
\end{itemize}

Then, there exists a Cantor set ${\mathcal{C}}\subset (-1,1)$ containing $0$%
, such that for every $\varepsilon \in {\mathcal{C}}$ the map $%
f_{\varepsilon }$ is smoothly conjugated to a rotation $R_{\alpha
_{\varepsilon }}$, with $\alpha _{\varepsilon }\in \mathcal{D}(\tau )$. More
specifically, there exists 
\begin{equation*}
h_{\varepsilon }(x)=x+\varepsilon v(x,\varepsilon )\in C^{\infty }({\mathbb{S%
}}^{1})
\end{equation*}%
such that 
\begin{equation}
\begin{array}{ccc}
{{\mathbb{S}}^{1}} & \overset{f_{\varepsilon }}{\longrightarrow } & {{%
\mathbb{S}}^{1}} \\ 
\uparrow {\small h_{\varepsilon }} &  & \uparrow {\small h_{\varepsilon }}
\\ 
{{\mathbb{S}}^{1}} & \overset{R_{\alpha _{\varepsilon }}}{\longrightarrow }
& {{\mathbb{S}}^{1}}%
\end{array}%
\qquad \Longleftrightarrow \qquad f_{\varepsilon }\circ h_{\varepsilon
}=h_{\varepsilon }\circ R_{\alpha _{\varepsilon }}.  \label{conjugation}
\end{equation}%
Moreover:

\begin{itemize}
\item the maps $\varepsilon \longmapsto h_{\varepsilon }$ and $\varepsilon
\longmapsto \alpha _{\varepsilon }$ are $C^{\infty }$ on the Cantor set ${%
\mathcal{C}}$, in the sense of Whitney;

\item $\alpha_{\varepsilon} = \alpha + A\varepsilon^{m+1} +
o(\varepsilon^{m+1}). $\newline
\end{itemize}
\end{theorem}

\medskip

\begin{remark}
\textrm{\label{rm8} Observe that $f_{\varepsilon}$ does not have necessarily
rotation number $\alpha$. In particular, the map $rot:\mathrm{Diff}_{+}^{0}(%
\mathbb{S}^{1})\longrightarrow $}$\mathbb{S}^{1}$\textrm{\ is continuous
with respect to the $C^{0}$-topology (see for example \cite[Ch. II,
Proposition 2.7]{Herman})}
\end{remark}

\begin{remark}
\label{remarkteokam} \hspace{0.1 cm}\newline
\begin{itemize}
\item[\textrm{(i)}] Theorem \ref{KAMVano} is proved in \cite{Vano} in a more
general form, considering also the cases of $u(x,\varepsilon)$ being
analytic or just finitely differentiable (in this case, there is a lower
bound on the needed differentiablity, cf. Theorem \ref{thmhermanyoccoz}). In
particular, the proof of the asymptotic expansion of $\alpha_{\varepsilon}$
appears on \cite[p. 149]{Vano}.

\item[\textrm{(ii)}] One could provide an estimate of the size of this
Cantor set: {\ there exist $M>0$ and $r_0>0$ such that for all $0<r<r_0$ the
set $(-r,r)\cap {\mathcal{C}} $ has lebesgue measure $\geq M r^{\frac{1}{m+1}%
}$ } (see \cite[formula (9.2)]{Vano}).

\item[\textrm{(iii)}] A version of this theorem in the analytic case, can be
also found in \cite[Theorem 2]{Arnold}; in particular, in \cite[Sections 8]%
{Arnold} it is discussed the property of monogenically dependence of the
conjugacy and the rotation number on the parameter.\newline
These results can be extended to arbitrary smooth circle diffeomorphisms
with Diophantine rotation numbers and to higher dimensional tori (see \cite%
{Vano}).
\end{itemize}
\end{remark}

\medskip

Let us discuss how to deduce from this result the existence of linear
response for the circle diffeomorphisms $f_{\varepsilon }$.\newline

\begin{theorem}
\label{KAMandResp} Let $\alpha \in \mathcal{D}({\tau})$, with $\tau>1$ and
let us consider a family of circle diffeomorphisms obtained by perturbing
the rotation $R_\alpha$ in the following way: 
\begin{equation*}
f_\varepsilon(x) = R_\alpha + \varepsilon u(x,\varepsilon) \qquad
|\varepsilon|< 1,
\end{equation*}
where $u(x,\varepsilon) \in C^{\infty}({\mathbb{S}}^1)$, for every $%
|\varepsilon|<1$, and the map $\varepsilon \longmapsto u(\cdot, \varepsilon)$
is $C^{\infty}$.\newline
Then, the circle rotation $R_\alpha$ admits linear response, in the limit as 
$\varepsilon$ goes to $0$, by effect of this family of perturbations.\newline
More precisely, there exists a Cantor set $\mathcal{C}\subset (-1,1)$ such
that 
\begin{equation}  \label{limitlinearresponse}
\lim_{\varepsilon \in \mathcal{C}, \varepsilon \rightarrow 0} \frac{%
\mu_\varepsilon - m}{\varepsilon} = 2\pi i \sum_{n \in \mathbb{Z }\setminus
\{0\}} \left(\frac{n\, \hat{u}(n)}{1- e^{2\pi i n \alpha}}\right) e^{2\pi i
n x} \qquad \mbox{(in the $L^1$-sense)}
\end{equation}
where $\mu_\varepsilon$ denotes the unique invariant probability measure of $%
f_\varepsilon$, for $\varepsilon \in {\mathcal{C}}$, and $\{\hat{u}%
(n)\}_{n\in {\mathbb{Z}}}$ the Fourier coefficients of $u(x,0)$.\newline
\end{theorem}

\medskip

\begin{remark}
\label{remarkKAM} In this article we focus on the circle;
however, a similar result could be proved for rotations on higher
dimensional tori, by using analogous KAM results in that setting (see for
example \cite{Vano}).
\end{remark}

\medskip

As we have already observed in Remark \ref{rm8}, the rotation number of $%
f_{\varepsilon}$ varies continuously with respect to the perturbation, from
here the need of taking the limit in \eqref{limitlinearresponse} on a Cantor
set of parameters (corresponding to certain Diophantine rotation numbers {for which the KAM algorithm can be applied}).
 { Under the assumption that the perturbation does not change the rotation number, {and this is Diophantine},
then the KAM algorithm can be applied for all values of the parameters $\varepsilon$, hence $\mathcal{C}$ coincides with the whole set of parameters; 
therefore the  limit in \eqref{limitlinearresponse} can be taken in the classical sense.}

\begin{corollary}
\label{corKAM} Under the same hypotheses and notation of Theorem \ref%
{KAMandResp}, if in addition we have that $\mathrm{rot}(f_\varepsilon) =
\alpha$ for every $|\varepsilon|<1$, then there exists linear response
without any need of restricting to a Cantor set and it is given by 
\begin{equation}
\lim_{\varepsilon \rightarrow 0} \frac{\mu_\varepsilon - m}{\varepsilon} =
2\pi i \sum_{n \in \mathbb{Z }\setminus \{0\}} \left(\frac{n\, \hat{u}(n)}{%
1- e^{2\pi i n \alpha}}\right) e^{2\pi i n x} \qquad 
\mbox{(in the
$L^1$-sense)}.
\end{equation}
\end{corollary}

\bigskip

{
\begin{proof} {\bf (Corollary \ref{corKAM}).}
As we have remarked above, this corollary easily follows from Theorem \ref{KAMandResp} by observing that 
$\mathrm{rot}(f_\varepsilon) = \alpha   \in \mathcal{D}({\tau})$ for every $|\varepsilon|<1$, hence
$\mathcal{C} \equiv (-1,1)$. In fact, this follows from  \cite[Section 9.2, pp. 147-148]{Vano}: in their notation our parameter $\varepsilon$ corresponds to $\mu$ and their $a(\mu)$ corresponds to our $\mathrm{rot}(f_\varepsilon)$.
In particular, they define the Cantor set as  ${\mathcal C}_F = v^{-1}(D_\Upsilon)$ (see \cite[p.148]{Vano}): in our notation this corresponds to
the values of $\varepsilon \in (-1,1)$ for which $\mathrm{rot}(f_\varepsilon)$ belongs to the a certain set of  Diophantine numbers  that includes  $\alpha$. Since, by hypothesis, $\mathrm{rot}(f_\varepsilon)\equiv \alpha$, it follows that
 ${\mathcal C}\equiv (-1,1)$ and, in particular, the  limit in \eqref{limitlinearresponse} is meant in the classical sense.
\end{proof}
}
\bigskip

Let us now prove Theorem \ref{KAMandResp}.\newline

\begin{proof} {\bf (Theorem \ref{KAMandResp}).} 
First of all, applying Theorem \ref{KAMVano}, it follows that for every $%
\varepsilon \in {\mathcal{C}}$, the map $f_{\varepsilon }:= R_\alpha +
\varepsilon u(x,\varepsilon)$ possesses a unique invariant probability
measure given by 
\begin{equation*}
\mu _{\varepsilon }={h_{\varepsilon }}_{\ast }m
\end{equation*}%
where $m$ denotes the Lebesgue measure on ${{\mathbb{S}}^{1}}$ and ${%
h_{\varepsilon }}_{\ast }$ denotes the push-foward by $h_{\varepsilon }$; in
particular, $\mu _{0}=m$. This measure is absolutely continuous with respect
to $m $ and its density is given by 
\begin{equation}
\frac{d\mu _{\varepsilon }}{dx}(x)=\frac{1}{\partial _{x}h_{\varepsilon
}(h_{\varepsilon }^{-1}(x))}.  \label{density}
\end{equation}%
In fact, if $A$ is a Borel set in ${{\mathbb{S}}^{1}}$, then 
\begin{equation*}
\mu _{\varepsilon }(A)=\int_{A}\mu _{\varepsilon }(dy)=\int_{h_{\varepsilon
}(A)}\partial _{x}(h_{\varepsilon }^{-1})(x)\,dx=\int_{h_{\varepsilon }(A)}%
\frac{dx}{\partial _{x}h_{\varepsilon }(h_{\varepsilon }^{-1}(x))}.
\end{equation*}

Hence, it follows from \eqref{density} that 
\begin{eqnarray}  \label{densitymueps}
\frac{d\mu _{\varepsilon }}{dx}(x) &= & \frac{1}{\partial_x h_{\varepsilon}
(h_{\varepsilon}^{-1}(x))} = \frac{1}{1 + \varepsilon \partial_x v
(h_{\varepsilon}^{-1}(x),0) + o(\varepsilon)}  \notag \\
&=& \frac{1}{1 + \varepsilon \partial_x v(x,0) + o_{\mathcal{C}}(\varepsilon)%
} = 1-\varepsilon \partial_x v(x,0) + o_{\mathcal{C}}(\varepsilon),
\end{eqnarray}
where $o_{\mathcal{C}}(\varepsilon)$ denotes a term that goes to zero faster
than $\varepsilon \in {\mathcal{C}}$, uniformly in $x$.\newline

Then the linear response is given by%
\begin{equation*}
\dot{\mu}=\lim_{\varepsilon \in {\mathcal{C}},\varepsilon \rightarrow 0}%
\frac{\mu _{\varepsilon }-\mu _{0}}{\varepsilon }=\lim_{\varepsilon \in {%
\mathcal{C}},\varepsilon \rightarrow 0}\frac{\mu _{\varepsilon }-m}{%
\varepsilon }
\end{equation*}%
which, passing to densities and using \eqref{densitymueps}, corespond to%
\begin{equation*}
\lim_{\varepsilon \in {\mathcal{C}},\varepsilon \rightarrow 0}\frac{1}{%
\varepsilon }(1-\varepsilon \partial _{x}v(x,0)+o_{0}(\varepsilon
)-1)=-\partial _{x}v(x,0).
\end{equation*}

Giving a formula for the response%
\begin{equation}  \label{linearresponse}
\frac{d\dot{\mu}}{dx}(x)= - \partial_x v (x,0).
\end{equation}
\medskip

Moreover, we can find a more explicit representation formula
{(the above formula, in fact, is somehow implicit, since $v$ depends on $h_\varepsilon$)}. Observe that it
follows from \eqref{conjugation} that $f_\varepsilon \circ h_\varepsilon =
h_\varepsilon \circ R_{\alpha_\varepsilon}$: 
\begin{equation}  \label{boh}
x + \varepsilon v(x,\varepsilon) + \alpha + \varepsilon u(x + \varepsilon
v(x,\varepsilon), \varepsilon) = x + \alpha_\varepsilon + \varepsilon
v(x+\alpha_\varepsilon,\varepsilon).
\end{equation}

Recall, from the statement of Theorem \ref{KAMVano} that 
\begin{equation*}
\alpha_{\varepsilon} = \alpha + A\varepsilon^{m+1} + o(\varepsilon^{m+1}),
\end{equation*}
where $m$ and $A$ are defined by (see item (ii) in Theorem \ref{KAMVano}) 
\begin{equation*}
<u(\cdot,\varepsilon)>:=\int_{{\mathbb{S}}^1} u(x,\varepsilon) dx =
A\varepsilon^m + o(\varepsilon^m).
\end{equation*}

Hence, expanding equation \eqref{boh} in terms of $\varepsilon$ and equating
the terms of order $1$, we obtain the following (observe that $%
\alpha_\varepsilon$ will contribute to the first order in $\varepsilon$ only
if $m=0$ and, therefore, $A= <u(\cdot,0)>:= \int_{{\mathbb{S}}^1} u(x,0) dx
\neq 0$):

\begin{equation}  \label{homologicaleq}
v(x+\alpha, 0) - v(x,0) = u(x,0) - <u(\cdot,0)> \qquad \forall \,x\, \in {{%
\mathbb{S}}^1},
\end{equation}
the so-called \textit{homological equation}.

Observe that it makes sense that we need to subtract to $u(x,0)$ its
average, if this is not zero. In fact, in order for \eqref{homologicaleq} to
have a solution, its right-hand side must have zero average: to see this, it
is sufficient to integrate both sides and use that the Lebesgue measure is
invariant under $R_\alpha$: 
\begin{equation*}
\int_{{\mathbb{S}}^1} u(x,0) \, dx = \int_{\mathbb{S}^1} v(x+\alpha,0) \, dx
- \int_{\mathbb{S}^1} v(x,0) \, dx =0.
\end{equation*}

Let us now find an expression for $v(x,0)$ in Fourier series. In fact, let
us consider: 
\begin{equation*}
v(x,0):= \sum_{n\in \mathbb{Z}} \hat{v}(n) e^{2\pi i n x} \qquad \mathrm{and}
\qquad u(x,0):= \sum_{n\in \mathbb{Z}} \hat{u}(n) e^{2\pi i n x}.
\end{equation*}

In Fourier terms, \eqref{homologicaleq} becomes: 
\begin{equation*}
\sum_{n\in \mathbb{Z}} \hat{v}(n) \left( e^{2\pi i n \alpha} -1 \right) \,
e^{2\pi i n x} = \sum_{n\in \mathbb{Z }\setminus \{0\}} \hat{u}(n) e^{2\pi i
n x}
\end{equation*}
and therefore for $n\neq 0$ 
\begin{equation*}
\hat{v}(n) = \frac{\hat{u}(n)}{e^{2\pi i n \alpha} -1};
\end{equation*}
we do not determine $\hat{v}(0)$, as it should be expected, since $v$ is
determined by \eqref{homologicaleq} only up to constants.

Substituting in \eqref{linearresponse}, we conclude: 
\begin{eqnarray*}
\frac{d \dot{\mu}}{dx}(x) &=& - \partial_x v (x,0) = - 2\pi i \sum_{n\in 
\mathbb{Z }} \,n \,\hat{v}(n) e^{2\pi i n x} \\
&=& 2\pi i \sum_{n \in \mathbb{Z }\setminus \{0\}} \left(\frac{n\, \hat{u}(n)%
}{1- e^{2\pi i n \alpha}}\right) e^{2\pi i n x}.
\end{eqnarray*}
\end{proof}

\section{Beyond rotations: the case of circle diffeomorphisms \label{sec:stabdiff} %
}
{In this section, we want to describe how it is possible to}
extend our previous results from irrational rotations to 
diffeomorphisms of the circle having irrational rotation number.

 We prove
the following:

\begin{theorem}
\label{stadiff}Let $T_{0}$ be an orientation preserving diffeomorphism of
the circle with an irrational rotation number $\alpha $ and such that $\log
(T^{\prime })$ has bounded variation (for example f is of class $C^{2}$).
{Let $\mu_0$ be its unique invariant (absolutely continuous) probability measure (see Theorem \ref{Denteo})}.\
Let $\{T_{\delta }\}_{0\leq \delta \leq \overline{\delta }}$ be a family of
Borel measurable maps of the circle such that%
\begin{equation*}
\sup_{x\in {\mathbb S}^{1}}|T_{0}(x)-T_{\delta }(x)|\leq \delta .
\end{equation*}%
Suppose that for each $0\leq \delta \leq \overline{\delta }$, $\mu _{\delta
} $ is an invariant measure of $T_{\delta }$. Then 
\begin{equation*}
\lim_{\delta \rightarrow 0}\int_{{\mathbb S}^{1}} f~d\mu _{\delta }=\int_{{\mathbb S}^{1}} f~d\mu _{0}
\end{equation*}%
for all $f\in C^{0}(\mathbb{S}^{1}).$
\end{theorem}

{The proof will follow by combining Theorem \ref{statstab} with Denjoy Theorem \ref{Denteo}.\\
}

\begin{proof}[Proof of Theorem \protect\ref{stadiff}]
By Theorem \ref{Denteo} we can coniugate $T_{0}$ with the rotation $%
R_{\alpha }.$ We apply the same coniugation to $T_{\delta }$ for each $%
\delta >0$ obtaining a family of maps {$U_{\delta }:= h \circ T_\delta \circ h^{-1}$}.
 We summarize the
situation in the following diagram%

\begin{equation}
\begin{array}{ccc}
{{\mathbb{S}}^1} & \overset{T_0}{\longrightarrow } & {{\mathbb{S}}^1} \\ 
\downarrow {\small h} &  & \downarrow {\small h} \\ 
{{\mathbb{S}}^1} & \overset{R_{\alpha}}{\longrightarrow } & {{%
\mathbb{S}}^1}%
\end{array}%
\qquad%
\begin{array}{ccc}
{{\mathbb{S}}^1} & \overset{T_\delta}{\longrightarrow } & {{\mathbb{S}}^1} \\ 
\downarrow {\small h} &  & \downarrow {\small h} \\ 
{{\mathbb{S}}^1} & \overset{U_\delta}{\longrightarrow } & {{%
\mathbb{S}}^1}%
\end{array}%
\label{diagrams}
\end{equation}%

Since $h$ is an homeomorphism of a compact space it is uniformly continuous.
This implies that 
\begin{equation*}
\lim_{\delta \rightarrow 0}\sup_{x\in {\mathbb S}^{1}}|R_{\alpha }(x)-U_{\delta
}(x)|=0.
\end{equation*}%
Let $\overline{\mu }_{\delta }:=h_{\ast }\mu _{\delta }.$ These measures
are invariant for $U_{\delta }.$\ \ Then, by Theorem \ref{statstab} we  get%
\begin{equation*}
\lim_{\delta \rightarrow 0}||\overline{\mu }_{\delta }-m||_{W}=0.
\end{equation*}%
This implies (uniformly approximating any continuous fuction with a sequence
of Lipschitz ones) that for each $g\in C^{0}(\mathbb{S}^{1})$%
\begin{equation}
\lim_{\delta \rightarrow 0}\int_{\mathbb S^1} g~d\overline{\mu }_{\delta }=\int_{\mathbb S^1} g~dm.
\label{inte}
\end{equation}%
Now consider $f\in C^{0}(\mathbb{S}^{1})$ and remark that {(using the definition of push-forward of a measure)}  
\begin{eqnarray*}
\int_{\mathbb S^1} f~~d\mu _{\delta } &=&\int_{\mathbb S^1} f\circ h^{-1} \circ h~d\mu _{\delta }=\int_{\mathbb S^1}
f\circ h^{-1}~d\overline{\mu }_{\delta }, \\
\int_{\mathbb S^1} f~d\mu _{0} &=&\int_{\mathbb S^1} f\circ h^{-1}~d\overline{\mu }_{0}.
\end{eqnarray*}%
By \ref{inte}, considering $g=f\circ h^{-1}$ this shows 
\begin{equation*}
\lim_{\delta \rightarrow 0}\int_{\mathbb S^1} f~d\mu _{\delta }=\int_{\mathbb S^1} f~d\mu _{0}.
\end{equation*}
\end{proof}

\bigskip

{Similarly, one can extend the quantitative stability results proved in  Theorem \ref{stst2} to smooth diffeomorphisms of the circle}.

{
\begin{remark}
We point out that the following theorem  holds under much less regularity for $T_0$ (the proof remains the same). In fact,  it is enough that $T_0\in C^r({\mathbb S^1})$ with 
$r$ sufficiently big so that the cojugation $h$ is bi-Lipschitz; compare with Theorem \ref{thmhermanyoccoz}. 
\end{remark}
}

\medskip

\begin{theorem}
\label{quantdiff}Let $T_{0}$ be a $C^{\infty }$ diffeomorphism of the circle
with Diophantine rotation number $\alpha \in \mathcal{D}(\tau )$, for some $\tau>1$. Let $%
\{T_{\delta }\}_{0\leq \delta \leq \overline{\delta }}$ be a family of Borel
measurable maps of the circle such that%
\begin{equation*}
\sup_{x\in {\mathbb S^{1}}}|T_{0}(x)-T_{\delta }(x)|\leq \delta .
\end{equation*}%
Suppose that for each $0\leq \delta \leq \overline{\delta }$, $\mu _{\delta
} $ is an invariant measure of $T_{\delta }$. Then, for each $\ell <{\frac{1%
}{\gamma (\alpha )+1}}$ we have: 
\begin{equation*}
\Vert m-\mu _{\delta }\Vert _{W}=O(\delta ^{\ell }).
\end{equation*}
\end{theorem}

\begin{proof}
By Theorem \ref{thmhermanyoccoz} , there exists $h\in \mathrm{Diff}%
_{+}^{\infty }({{\mathbb{S}}^{1}})$ conjugating $T_{0}$ with the rotation $%
R_{\alpha }.$ We apply the same coniugation to $T_{\delta }$ for each $%
\delta >0$ obtaining a family of maps $U_{\delta }.$ \ The situation is
still summarized by $(\ref{diagrams}).$ Since $h$ is a bilipschitz map we
have 
\begin{equation*}
\lim_{\delta \rightarrow 0}\sup_{x\in {\mathbb S^{1}}}|R_{\alpha }(x)-U_{\delta }(x)|=0
\end{equation*}%
and there is a $C\geq 1$ such that for any pair of probability measures $\mu
_{1},\mu _{2}$%
\begin{equation*}
C^{-1}||\mu _{1}-\mu _{2}||_{W}\leq ||h_{\ast }^{-1}\mu _{1}-h_{\ast
}^{-1}\mu _{2}||_{W}\leq C||\mu _{1}-\mu _{2}||_{W}
\end{equation*}%
(and the same holds for $h_{\ast }$). Let $\overline{\mu }_{\delta
}:=h_{\ast }(\mu _{\delta }).$ These measures are invariant for $U_{\delta
}. $\ \ 

By Theorem \ref{stst2} we then get that for each $\ell <{\frac{1}{\gamma
(\alpha )+1}}$ we have: 
\begin{equation*}
\Vert m-\overline{\mu }_{\delta }\Vert _{W}=O(\delta ^{\ell }).
\end{equation*}%
This imply 
\begin{equation*}
\Vert \mu _{0}-\mu _{\delta }\Vert _{W}=||h_{\ast }^{-1}m-h_{\ast }^{-1}%
\overline{\mu }_{\delta }||_{W}=O(\delta ^{\ell }).
\end{equation*}
\end{proof}

\bigskip

{Finally, one can also extend the existence of linear response, along the same lines of Theorem \ref{KAMandResp} and Corollary \ref{corKAM}.
In fact, as observe in Remark \ref{remarkteokam} ({\it iii}), KAM theorem can be extended to sufficiently regular diffeomorphisms of the circle (one can prove it either directly ({\it e.g.}, \cite{Arnold, BroerSevryuk, Moser, Russmann,Vano}), or 
by combining the result for rotations of the circle, with Theorem \ref{thmhermanyoccoz}).
Since the proof can be adapted {\it mutatis mutandis} {(of course, leading to a different expression for the linear response}), we omit further details.}\\

\medskip

\section{Stability under discretization and numerical truncation}\label{sectrunc}

As an application of what discussed in this section we want to address the
following question: 

\medskip

\noindent \textbf{Question:} \emph{Why are numerical simulations generally
quite reliable, in spite of the fact that numerical truncations are quite
bad perturbations, transforming the system into a piecewise constant one,
having only periodic orbits?}

\medskip

Let us consider the uniform grid $E_{N}$ on $\mathbb{S}^{1}$ defined by%
\begin{equation*}
E_{N}=\left\{\frac{i}{N}\in \mathbb{R}/\mathbb{Z}: \quad 1\leq i\leq N
\right\}.
\end{equation*}%
In particular when $N=10^{k}$ the grid represents the points which are
representable with $k$ decimal digits. Let us consider the projection $P_{N}:%
\mathbb{S}^{1}\rightarrow E_{N}$ defined by

\begin{equation*}
P_{N}(x)=\frac{\left\lfloor Nx\right\rfloor }{N},
\end{equation*}
where $\lfloor \cdot \rfloor$ is the floor function.

Given a map $T:$ $\mathbb{S}^{1}\rightarrow \mathbb{S}^{1}$ and let $N\in {%
\mathbb{N}}$; we define its \textit{$N$-discretization} $T_{N}:\mathbb{S}%
^{1}\rightarrow \mathbb{S}^{1}$ \ by%
\begin{equation*}
T_{N}(x):=P_{N}(T(x)).
\end{equation*}%
This is an idealized representation of what happens if we try to simulate
the behavior of $T$ on a computer, having $N$ points of resolution. Of
course the general properties of the systems $T_{N}$ and $T$ are a priori
completely different, starting from the fact that $T_{N}$ is forced to be
periodic. Still these simulations gives in many cases quite a reliable
picture of many aspects of the behavior of $T$, which justifies why these
naive simulations are still much used in many applied sciences.

 Focusing on the statistical properties of the system and on its invariant measures, one can investigate
whether the invariant measures of the system $T_{N}$ (when they exist)
converge to the physical measure of $T$, and in general if they converge to
some invariant measure of $T$. In this case, the statistical properties of $%
T $ are in some sense robust under discretization. Results of this kind have
been proved for some classes of pievewise expanding maps (see \cite{Bo}, 
\cite{GB})\ and for topologically generic diffeomorphisms of the torus (see 
\cite{Gu}, \cite{Gu2}, \cite{mier}).

Since the discretization is a small perturbation in the uniform convergence
topology, a direct application of Theorem \ref{stadiff} gives

\begin{corollary}
\label{discrediffeo}Let $T_{0}$ be an orientation preserving 
diffeomorphism of the circle with an irrational rotation number $\alpha $
and such that $\log (T_{0}^{\prime })$ has bounded variation and let $N\geq
1 $. Let $T_{N}=P_{N}\circ T_{0}$ be the family of maps given by its $%
N-discretizations$. Suppose $\mu _{N}$ is an invariant measure of $T_{N}$.
Then 
\begin{equation*}
\lim_{N\rightarrow \infty }\int_{\mathbb S^1} f~d\mu _{N}=\int_{\mathbb S^1} f~d\mu _{0}
\end{equation*}%
for all $f\in C^{0}(\mathbb{S}^{1}).$
\end{corollary}

\begin{proof}
The statement follows by Theorem \ref{stadiff} noticing that 
\begin{equation*}
\sup_{x\in {\mathbb{S}}^{1}}|T_{0}(x)-T_{N}(x)|\leq \frac{1}{N}.
\end{equation*}
\end{proof}

We think this result is very similar to the one shown in Proposition 8.1 of \cite{mier}.
Comparing this kind of results with the ones in \cite{Gu}, we point out that in this
statement we do not suppose the system to be topologically generic and that
the convergence is proved for all discretizations, while in \cite{Gu} the
convergence is proved for a certain sequence of finer and finer
discretizations. \newline

As an application of our quantitative stability result (Theorem \ref{stst2}
and \ref{quantdiff}), we can also provide a quantitative estimate for the
speed of convergence of the invariant measure of the $N$-discretized system
to the original one. We remark that as far as we know, there are no other
similar quantitative convergence results of this kind in the literature.%
\newline

\begin{corollary}
\label{quantdiscrediffeo}Let $T_{0}$ be a $C^{\infty }$ diffeomorphism of
the circle with Diophantine rotation number $\alpha \in \mathcal{D}(\tau ).$
Let $T_{N}=P_{N}\circ T_0$ be the family of its $N$-discretizations.
Suppose $\mu _{N}$ is an invariant measure of $T_{N}$. Then, for each $\ell <%
{\frac{1}{\gamma (\alpha )+1}}$ 
\begin{equation*}
\Vert m-\mu _{N}\Vert _{W}=O(N^{-\ell }).
\end{equation*}
\end{corollary}

The proof of Corollary \ref{quantdiscrediffeo} \ is {similar to}
the one of Corollary \ref{discrediffeo}.

\end{document}